\documentclass[12pt,twoside,reqno]{amsart}
\usepackage{amsmath}
\usepackage{amsfonts}
\usepackage{amssymb}
\usepackage{color}
\usepackage{mathrsfs}
\usepackage{cite}
\newtheorem{theorem}{Theorem}[section]
\newtheorem{corollary}[theorem]{Corollary}
\newtheorem{lemma}[theorem]{Lemma}

\newtheorem{definition}[theorem]{Definition}
\newtheorem{remark}[theorem]{Remark}
\numberwithin{equation}{section}
\begin{document}
\title{Mass-conserving solutions to coagulation-fragmentation equations with balanced growth} 
\author{Philippe Lauren\c{c}ot}
\address{Institut de Math\'ematiques de Toulouse, UMR~5219, Universit\'e de Toulouse, CNRS \\ F--31062 Toulouse Cedex 9, France}
\email{laurenco@math.univ-toulouse.fr}

\keywords{coagulation, fragmentation, mass conservation}
\subjclass{45K05}

\date{\today}

\begin{abstract}
A specific class of coagulation and fragmentation coefficients is considered for which the strength of the coagulation is balanced by that of the multiple fragmentation. Existence and uniqueness of mass-conserving solutions are proved when the initial total mass is sufficiently small. 
\end{abstract}

\maketitle

%
%
\pagestyle{myheadings}
\markboth{\sc{Philippe Lauren\c cot}}{\sc{Mass-conserving solutions to C-F equations with balanced growth}}

\section{Introduction}\label{sec1}

The time evolution of the size distribution function of a population of particles growing by successive mergers and spontaneously breaking apart into smaller pieces is described by the coagulation-fragmentation equation
\begin{subequations}\label{a1}
\begin{align}
\partial_t f(t,x) & = \mathcal{C}f(t,x) + \mathcal{F}f(t,x)\ , \qquad (t,x) \in (0,\infty)^2\ , \label{a1a} \\
f(0,x) & = f^{in}(x)\ , \qquad x\in (0,\infty)\ , \label{a1d}
\end{align}
where the coagulation term $\mathcal{C}f$ and the fragmentation term $\mathcal{F}f$ are given by
\begin{equation}
\mathcal{C}f(x) := \frac{1}{2} \int_0^x K(y,x-y) f(x-y) f(y)\ \mathrm{d}y - \int_0^\infty K(x,y) f(x) f(y)\ \mathrm{d}y\ , \qquad x>0\ , \label{a1b}
\end{equation}
and 
\begin{equation}
\mathcal{F}f(x) := - a(x) f(x) + \int_x^\infty a(y) b(x,y) f(y)\ \mathrm{d}y\ , \qquad x>0\ , \label{a1c}
\end{equation}
\end{subequations}
respectively. In \eqref{a1}, $f(t,x)$ denotes the size distribution function of particles of size $x\in (0,\infty)$ at time $t>0$ and the sizes of the particles vary as a consequence of coagulation and fragmentation reactions. In \eqref{a1b}, $K$ denotes the coagulation kernel which is a non-negative and symmetric function defined on $(0,\infty)^2$ and $K(x,y)=K(y,x)$ is the rate at which two particles with respective sizes $x$ and $y$ merge. The first term in \eqref{a1b} accounts for the formation of particles of size $x$ resulting from the coalescence of particles with respective sizes $y\in (0,x)$ and $x-y$ while the second term describes the disappearance of particles of size $x$ after coagulation with other particles. In \eqref{a1c}, $a(x)$ is the overall fragmentation rate of particles with size $x$ and the first term in \eqref{a1c} describes the loss of particles of size $x$ due to spontaneous breakage. The second term in \eqref{a1c} accounts for the appearance of new particles of size $x$ resulting from the break-up of a particle of size $y>x$ and $b(x,y)$ is the daughter distribution function providing the distribution of particles with size $x\in (0,y)$ resulting from the breakage of a particle with size $y$. We only consider here the case where there is no loss of matter during splitting; that is, 
\begin{equation}
\int_0^y x b(x,y)\ \mathrm{d}x = y\ , \qquad y>0\ . \label{a2}
\end{equation}
In fact, there is also no loss of matter during coalescence events, so that conservation of matter is expected throughout time evolution. From a mathematical point of view, this means that the first moment $M_1(f(t))$ of the distribution function $f(t)$ at time $t>0$, which somehow corresponds to the total mass of the system of particles, remains unchanged as time goes by; that is, 
\begin{equation}
M_1(f(t)) := \int_0^\infty x f(t,x)\ \mathrm{d}x = \varrho = M_1(f^{in}) := \int_0^\infty x f^{in}(x)\ \mathrm{d}x\ , \qquad t\ge 0\ . \label{a3}
\end{equation}
It is however by now well-known that the conservation of matter may fail to be true and there are basically two mechanisms responsible for the infringement of \eqref{a3}. On the one hand, for coagulation kernels increasing rapidly with size, a runaway growth might occur, thereby producing particles with infinite size. Since the distribution function $f$ does not include such particles, matter escapes from the system of particles with finite size to form giant clusters and the first moment of $f$ decays after some time. This phenomenon is known as \textit{gelation} and takes place when $K(x,y)$ increases faster than $\sqrt{xy}$ as $x+y\to \infty$ and the fragmentation is weak enough \cite{BLLxx, EMP02, ELMP03, Jeon98, Laur00, Leyv83, LeTs81}. On the other hand, if there is no coagulation and the overall fragmentation rate $a(x)$ is unbounded as $x\to 0$, small particles break apart at a high rate and thus reduce to particles with size zero or dust which are also not accounted for in the distribution function $f$. Another loss of matter occurs in that case and is known as \textit{shattering} \cite{Bana06, BLLxx, Fili61, McZi87}. To be more specific, when the coagulation and fragmentation coefficients are given by
\begin{subequations}\label{a4}
\begin{equation}
K(x,y) = K_0 \left( x^\alpha y^{\lambda-\alpha} + x^{\lambda-\alpha} y^\alpha \right)\ , \qquad (x,y)\in (0,\infty)^2\ , \label{a4a}
\end{equation}
with $\alpha\in [0,1]$, $\lambda\in [2\alpha,1+\alpha]$, and $K_0 > 0$, and
\begin{equation}
a(x) = a_0 x^\gamma\ , \quad b(x,y) = (\nu+2) x^\nu y^{-\nu-1}\ , \qquad 0<x<y\ , \label{a4b}
\end{equation}
\end{subequations}
with $\gamma\in\mathbb{R}$, $\nu\in (-2,\infty)$, and $a_0> 0$, the following results are available: on the one hand, shattering is observed when $\gamma<0$ in \eqref{a4b} and there is no coagulation $(K_0=0)$ \cite{Bana06, Fili61, McZi87}. On the other hand, the gelation phenomenon takes place when $\alpha>1/2$ in \eqref{a4a}, $\gamma\in (0,\lambda-1)$ in \eqref{a4b}, and the total mass $\varrho=M_1(f^{in})$ of the initial condition is sufficiently large compared to the parameter $a_0$ in \eqref{a4b} \cite{ELMP03, EMP02, Laur00, ViZi89}. Mass-conserving solutions exist when $\gamma\ge 0$ and, either $\lambda\in [0,1]$, or $\lambda\in (1,2]$ and $\gamma > \lambda-1$ \cite{BaCa90, daCo95, ELMP03, EMRR05, LaMi02b, Stew89, Whit80}, which is in accordance with the numerical simulations and formal computations performed in \cite{Pisk12b, ViZi89}. Let us point out that the above mentioned results do not depend on the daughter distribution function $b$ when only a finite number of fragments are formed during breakup; that is, when $x\mapsto b(x,y)\in L^1(0,y)$ for all $y>0$. This corresponds to the choice $\nu>-1$ in \eqref{a4b}. When $\nu\in (-2,-1]$, the fact that an infinite number of fragments are produced after breakup alters the situation \cite{LavR15, ViZi89}.

According to the previous discussion, the value $\gamma=\lambda-1>0$ appears to be a borderline case with respect to the occurrence of the gelation phenomenon. The purpose of this paper is to contribute to a better understanding of this particular case for which nothing much seems to be known, except for the particular choice of parameters
$$
\alpha=1\ , \qquad \lambda=2\ , \qquad \gamma=1\ .
$$
In that case, explicit computations performed in \cite{Pisk12b, ViZi89} reveal that gelation is likely to take place when $a_0/(\varrho K_0)$ is suitably small and $\nu>-1$ and a proof that it indeed occurs when $a_0/(\varrho K_0)<\nu+1$ may be found in \cite{BLLxx}. On the opposite, numerical simulations performed in \cite{Pisk12b} indicate that there are mass-conserving solutions when $\nu=0$ and the ratio $a_0/(\varrho K_0)$ is large enough. Still for $\nu=0$, it follows from \cite[Remark~5.2]{DLP17} that there is a stationary solution with total mass $K_0/a_0$. Finally, when $\nu=-1$, the initial value problem  \eqref{a1} has at least one mass-conserving solution whatever the values of $\varrho$, $K_0$, and $a_0$  in \eqref{a3} and \eqref{a4} and the value of $\varrho$ \cite{LavR15}. 

We shall here extend these results to a larger class of coagulation and fragmentation coefficients, showing that mass-conserving solutions exist for a suitable range of the values of the parameters $\lambda$ and $\nu$, provided the ratio $a_0/(\varrho K_0)$ is large enough. More precisely, consider 
\begin{subequations}\label{aaCFC}
\begin{equation}
\lambda \in (1,2]\ , \quad \gamma:= \lambda-1 \in (0,1]\ , \quad \alpha\in \left[ \max\left\{ \frac{1}{2} , \lambda-1 \right\} , \frac{\lambda}{2} \right]\ , \label{aa1}
\end{equation}
and assume that the overall fragmentation rate $a$ and the coagulation kernel $K$ are given by 
\begin{eqnarray}
a(x) & = & a_0 x^{\lambda-1}\ , \qquad x\in (0,\infty)\ , \label{aa2} \\
K(x,y) & = & K_0 \left( x^\alpha y^{\lambda-\alpha} + x^{\lambda-\alpha} y^\alpha \right)\ , \qquad (x,y)\in (0,\infty)^2\ , \label{aa3}
\end{eqnarray}
for some positive constants $a_0$ and $K_0$. We assume further that the daughter distribution function $b$ has the scaling form
\begin{equation}
b(x,y) = \frac{1}{y} B\left( \frac{x}{y} \right)\ , \qquad 0<x<y\ , \label{aa4}
\end{equation}
where
\begin{equation}
B\ge 0 \;\;\text{ a.e. in }\;\; (0,1)\ , \quad B\in L^1((0,1),z\mathrm{d}z)\ , \quad \int_0^1 zB(z)\ \mathrm{d}z =1\ , \label{aa5}
\end{equation}
and there is $\nu\in (-2,0]$ such that
\begin{equation}
\mathfrak{b}_{m,p} := \int_0^1 z^m B(z)^p\ \mathrm{d}z < \infty\ ,   \label{aa6}
\end{equation}
for all $(m,p)\in\mathcal{A}_\nu$, the set $\mathcal{A}_\nu$ being defined by
\begin{equation}
\mathcal{A}_\nu :=\{ (m,p)\in (-1,\infty)\times [1,\infty)\ :\ m+p\nu>-1\}\ . \label{aa7}
\end{equation}
The final assumption connects the small size behaviour of the coagulation kernel $K$ with the possible singularity of $B$ for small sizes. Specifically, we require
\begin{equation}
-\nu-1 < \alpha\ . \label{aa8}
\end{equation}
\end{subequations}

When 
\begin{equation}
B(z) = B_\nu(z) := (\nu+2) z^\nu\ , \qquad z\in (0,1)\ , \;\;\text{ for some }\;\; \nu\in (-2,0]\ , \label{Bnu}
\end{equation} 
the assumptions \eqref{aa5} and \eqref{aa6} are clearly satisfied for all $\nu\in (-2,0]$, while \eqref{aa8} requires $\nu>-\alpha-1$. Let us point out that \eqref{aa8} is implied by \eqref{aa1} when $\nu>-3/2$. 

\begin{remark}\label{rema1}
Clearly $(m,1)\in \mathcal{A}_\nu$ for all $m>-\nu-1$, so that $\mathcal{A}_\nu$ is non-empty. We also observe that, if $(m,1)\in\mathcal{A}_\nu$, then $(m,p)$ also belongs to $\mathcal{A}_\nu$ for all $p\in [1,(m+1)/|\nu|)$ when $\nu<0$ and for all $p\in [1,\infty)$ when $\nu=0$. 
\end{remark}

An easy consequence of \eqref{aa6} and Remark~\ref{rema1} is that
\begin{equation}
\mathfrak{b}_{\ln} := \int_0^1 z |\ln z| B(z)\ \mathrm{d}z < \infty\ . \label{aa9}
\end{equation}
Indeed, 
\begin{equation*}
\int_0^1 z |\ln z| B(z)\ \mathrm{d}z \le \sup_{z\in (0,1)}\left\{ z^{(2+\nu)/2} |\ln{z}| \right\} \int_0^1 z^{-\nu/2} B(z)\ \mathrm{d}z = \frac{2\mathfrak{b}_{-\nu/2,1}}{e(\nu+2)}\ ,
\end{equation*}
and the right-hand side of the previous inequality is finite since $(-\nu/2,1)\in\mathcal{A}_\nu$. We then set
\begin{equation}
\varrho_\star := \frac{a_0 \mathfrak{b}_{\ln}}{2 K_0 \ln{2}}\ , \label{aa10}
\end{equation}
and define the weighted $L^1$-space $X_m$ and the moment $M_m(h)$ of order $m$ of $h\in X_m$ for $m\in\mathbb{R}$ by
\begin{equation*}
X_m := L^1((0,\infty),x^m\mathrm{d}x)\ , \qquad M_m(h) := \int_0^\infty x^m h(x)\ \mathrm{d}x\ .
\end{equation*}
We also denote the positive cone of $X_m$ by $X_m^+$ while $X_{m,w}$ denotes the space $X_m$ endowed with its weak topology. 

With the above notation, the main contribution of this paper is the existence and uniqueness of mass-conserving weak solutions to \eqref{a1} on $[0,\infty)$ when $\varrho\in (0,\varrho_\star)$.

\begin{theorem}\label{ThmP1}
Let $K$, $a$, and $b$ be coagulation and fragmentation coefficients satisfying \eqref{aaCFC} and fix
\begin{equation}
m_0\in (-\nu-1,\alpha) \cap [0,1)\ . \label{aa11}
\end{equation}
Consider a non-negative initial condition $f^{in}\in X_1^+$ such that 
\begin{equation}
\int_0^\infty \left( x^{m_0} + x|\ln x| \right) f^{in}(x)\ \mathrm{d}x < \infty\ , \label{aa12}
\end{equation}
and assume that
\begin{equation}
\varrho := M_1(f^{in}) \in (0,\varrho_\star)\ . \label{a100}
\end{equation}
\begin{itemize} 
\item[(a)] There is a mass-conserving weak solution $f$ to \eqref{a1} on $[0,\infty)$ in the sense of Definition~\ref{DefS1} below, which additionally belongs to $L^1((0,T),X_\lambda)$ for all $T>0$. 
\item[(b)] If $f^{in}\in X_m$ for some $m\in (-\nu-1,m_0)\cup (1+\lambda-\alpha,\infty)$, then $f\in L^\infty((0,T),X_m)$ for all $T>0$,
\item[(c)] If $f^{in}$ belongs to $X_{2\lambda-\alpha}$, then there is a unique mass-conserving weak solution $f$ to \eqref{a1} on $[0,\infty)$ such that $f\in L^\infty((0,T),X_{2\lambda-\alpha})$ for all $T>0$.
\end{itemize}
\end{theorem}

The proof of Theorem~\ref{ThmP1}~(a)-(b) relies on a weak compactness method in a weighted $L^1$-space, an approach pioneered in \cite{Stew89} and further developed in \cite{ELMP03, EMRR05, GKW11, GLW12, Laur00, Laur18, LaMi02b}. The starting point of the analysis is to construct a sequence of approximations of \eqref{a1} for which the existence of a solution follows by an application of Banach fixed point theorem. To guarantee the $L^1$-weak compactness of the sequence of approximating solutions thus obtained, several estimates are to be derived. In particular, the behaviour of the sequence of the approximating solutions for small and large sizes has to be controlled and the balance between the homogeneity of the coagulation kernel $K$ and the overall fragmentation rate $a$ is the main difficulty to be overcome here. In fact, special care has to be paid to the control for large sizes and this is where the smallness condition on the total mass $\varrho$ comes into play. The key idea is here to obtain joint estimates in $L^1((0,\infty),x|\ln x|\mathrm{d}x)$ and in $X_{m_0}$ when the total mass of the initial condition does not exceed $\varrho_\star$. Uniform integrability and time equicontinuity are next derived, the former relying on the monotonicity properties of the coagulation kernel \cite{Buro83, Dubo94b, LaMi02c, MiRR03}, and complete the proof of the weak compactness of the sequence of approximating solutions. Passing to the limit as the approximating parameter diverges to infinity is then done along the lines of the above mentioned references. The uniqueness statement in Theorem~\ref{ThmP1}~(c) is proved along the lines of \cite[Theorem~2.9]{EMRR05}.

\begin{remark}\label{rema2}
Note that the value $\varrho_\star$ defined in \eqref{aa10} is finite whatever the value of $\nu$. On the one hand, that mass-conserving solutions cannot exist when $\varrho$ exceeds a positive threshold value $\varrho_{th}$ is supported by the occurrence of gelation for large values of $\varrho$ when $\lambda=2$, $\alpha=1$, and $B=B_\nu$ with $\nu\in (-1,0]$, see \eqref{Bnu}, \cite{BLLxx}. However, it is likely that $\varrho_{th}>\varrho_\star$: for instance, it is expected that $\varrho_{th}=a_0/K_0$ when $\alpha=1$, $\lambda=2$, and $B=B_0$ ($\nu=0$), and $\varrho_\star= a_0/(2K_0)$ in that case. On the other hand, the analysis performed in \cite{LavR15} reveals that gelation does not take place when $\lambda=2$, $\alpha=1$, and $B=B_{-1}$ ($\nu=-1$), so that the outcome of Theorem~\ref{ThmP1}~(a) is far from being optimal in that case. However, being valid for all values of $\nu$ in $(-2,0]$, the proof of Theorem~\ref{ThmP1} does not exploit the specific non-integrability features of $B_\nu$ when $\nu\in (-2,-1]$ and a more refined analysis seems to be required. We hope to return to that problem in the near future.
\end{remark}
 
\section{Weak Solutions: Existence} \label{sec2}

In this section, we assume that $K$, $a$, and $b$ are coagulation and fragmentation coefficients satisfying \eqref{aaCFC} and that the initial condition $f^{in}$ belongs to $X_1$ and possesses the integrability properties \eqref{aa12} and \eqref{a100} with $m_0$ satisfying \eqref{aa11}. We also fix 
\begin{equation}
m_1 \in [m_0,1) \cap [2-\lambda,1) \label{aa11b}
\end{equation}
and a positive real number $\sigma>0$ satisfying
\begin{equation}
\sigma \ge M_{m_0}(f^{in}) + M_1(f^{in}) + \frac{3}{e (1-m_1)} M_{m_1}(f^{in}) +  \int_0^\infty x|\ln x| f^{in}(x)\ \mathrm{d}x\ . \label{b00}
\end{equation}

Throughout this section, $C$ and $(C_i)_{i\ge 1}$ are positive constants depending only on $\lambda$, $\alpha$, $K_0$, $a_0$, $B$, $\nu$,  $m_0$, $m_1$, $\varrho$, and $\sigma$. Dependence upon additional parameters is indicated explicitly.

\newcounter{NumConst}

\subsection{Weak Solutions} \label{sec2.1}

We begin this section with the definition of a weak solution to \eqref{a1} which is slightly different from the usual one in order to account for the possible non-integrability of the daughter distribution $b$ \cite{BLLxx}. It involves the spaces of test functions $\Theta_m$, $m\in [0,1)$, which are defined by $\Theta_0:= L^\infty(0,\infty)$ and 
\begin{equation*}
\Theta_m := \left\{ \vartheta\in C^{0,m}([0,\infty))\cap L^\infty(0,\infty)\ :\ \vartheta(0) = 0 \right\}\ .
\end{equation*} 

\begin{definition}\label{DefS1}
Let $K$, $a$, and $b$ be coagulation and fragmentation coefficients satisfying \eqref{aaCFC} and consider $(m_0,m_1)$ satisfying \eqref{aa11} and \eqref{aa11b} and a non-negative function $f^{in}\in X_1^+\cap X_{m_0}$. For $T\in (0,\infty]$, a weak solution to \eqref{a1} on $[0,T)$ is a non-negative function 
\begin{equation*}
f\in C([0,T),X_{m_1,w})\cap L^\infty((0,T),X_{m_0}) \cap L^\infty((0,T),X_1)
\end{equation*} 
such that, for all $t\in (0,T)$ and $\vartheta\in\Theta_{m_1}$, 
\begin{align}
\int_0^\infty (f(t,x)-f^{in}(x)) \vartheta(x)\ \mathrm{d}x & = \frac{1}{2} \int_0^t \int_0^\infty \int_0^\infty K(x,y) \chi_\vartheta(x,y) f(s,x) f(s,y)\ \mathrm{d}y\mathrm{d}x\mathrm{d}s \nonumber\\
& \qquad - \int_0^t \int_0^\infty a(x) N_\vartheta(x) f(s,x)\ \mathrm{d}x\mathrm{d}s\ , \label{b0}
\end{align}
where
\begin{align}
\chi_\vartheta(x,y) & := \vartheta(x+y) -\vartheta(x) - \vartheta(y)\ , \qquad (x,y)\in (0,\infty)^2\ , \label{bchi}\\
N_\vartheta(y) & := \vartheta(y) - \int_0^y \vartheta(x) b(x,y)\ \mathrm{d}x\ , \qquad y>0\ . \label{bN}
\end{align}

In addition, $f$ is a mass-conserving weak solution to \eqref{a1} on $[0,T)$ if $f\in C([0,T),X_{1,w})$ and $M_1(f(t))=M_1(f^{in})$ for all $t\in [0,T)$.
\end{definition}

The existence of weak solutions and mass-conserving weak solutions to coagulation-fragmentation equations has been investigated for various classes of coagulation and fragmentation coefficients, see \cite{BaCa90, BaLa11, BaLa12a, BLL13, BLLxx, daCo95, Dubo94b, ELMP03, EMRR05, GKW11,  GLW12, Laur00, Laur18, LaMi02b, McLe62, Melz57b, Spou84, Stew89, Whit80} and the references therein. As for the uniqueness issue, we refer to \cite{BaCa90, BaLa11, BaLa12a, BLL13, BLLxx, daCo95, EMRR05, Giri13, GiWa11, LaMi04, Norr99, Stew90b} and the references therein.

\subsection{Approximation} \label{sec2.2}

We fix an initial condition $f^{in}\in X_1^+$ satisfying \eqref{aa12} and \eqref{a100} and recall that $\varrho=M_1(f^{in})$. Let $j\ge 2$ be an integer and define
\begin{align}
K_j(x,y) & := K(x,y) \mathbf{1}_{(0,j)}(x) \mathbf{1}_{(0,j)}(y)\ , \qquad (x,y)\in (0,\infty)^2\ , \label{b1} \\
a_j(x) & := a(x) \mathbf{1}_{(0,j)}(x)\ , \qquad x\in (0,\infty)\ , \label{b2}
\end{align}
and
\begin{equation}
f_j^{in}(x) := f^{in}(x) \mathbf{1}_{(0,j)}(x)\ , \qquad x\in (0,\infty)\ . \label{b3}
\end{equation}
Denoting the coagulation and fragmentation operators with $K_j$ and $a_j$ instead of $K$ and $a$ by $\mathcal{C}_j$ and $\mathcal{F}_j$, respectively, it follows from \eqref{aaCFC}, \eqref{b1}, and \eqref{b2} by a classical Banach fixed point argument that there is a unique non-negative function
\begin{equation*}
f_j\in C^1([0,\infty),L^1((0,j),x^{m_0}\mathrm{d}x))
\end{equation*}
which solves
\begin{subequations}\label{b4}
\begin{align}
\partial_t f_j(t,x) & = \mathcal{C}_j f_j(t,x) + \mathcal{F}_j f_j(t,x)\ , \qquad (t,x) \in (0,\infty)\times (0,j)\ , \label{b4a} \\
f_j(0,x) & = f_j^{in}(x)\ , \qquad x\in (0,j)\ . \label{b4b}
\end{align}
\end{subequations}
We extend $f_j$ to $[0,\infty)\times (0,\infty)$ by setting $f_j(t,x)=0$ for $t\ge 0$ and $x>j$. It readily follows from \eqref{b4a} that, for $t>0$ and $\vartheta\in C^{m_0}([0,2j])$ satisfying $\vartheta(0)=0$, 
\begin{align}
\frac{\mathrm{d}}{\mathrm{d}t} \int_0^j \vartheta(x) f_j(t,x)\ \mathrm{d}x & = \frac{1}{2} \int_0^j \int_0^j K(x,y) \chi_\vartheta(x,y) f_j(t,x) f_j(t,y)\ \mathrm{d}y\mathrm{d}x \nonumber\\
& \qquad - \int_0^j a(x) N_\vartheta(x) f_j(t,x)\ \mathrm{d}x \label{b5} \\
& \qquad - \frac{1}{2} \int_0^j \int_{j-y}^j K(x,y) \vartheta(x+y) f_j(t,x) f_j(t,y)\ \mathrm{d}x\mathrm{d}y\ . \nonumber
\end{align}
Choosing $\vartheta(x)=x$, $x\in (0,2j)$, in \eqref{b5} gives
\begin{equation*}
\frac{\mathrm{d}}{\mathrm{d}t} \int_0^j x f_j(t,x)\ \mathrm{d}x = - \frac{1}{2} \int_0^j \int_{j-y}^j (x+y) K(x,y) f_j(t,x) f_j(t,y)\ \mathrm{d}y\mathrm{d}x\ , 
\end{equation*}
hence, thanks to the symmetry of $K$,
\begin{equation*}
\frac{\mathrm{d}}{\mathrm{d}t} \int_0^j x f_j(t,x)\ \mathrm{d}x = - \int_0^j \int_{j-y}^j x K(x,y) f_j(t,x) f_j(t,y)\ \mathrm{d}y\mathrm{d}x\ .
\end{equation*}
After integration with respect to time, we find
\begin{equation*}
\int_0^j x f_j(t,x)\ \mathrm{d}x = \int_0^j x f_j^{in}(x)\ \mathrm{d}x - \int_0^t \int_0^j \int_{j-y}^j x K(x,y) f_j(s,x) f_j(s,y)\ \mathrm{d}y\mathrm{d}x\mathrm{d}s\ .
\end{equation*}
Recalling that $f_j(t,x)=0$ for $t\ge 0$ and $x\in (j,\infty)$, the previous identity becomes
\begin{equation}
M_1(f_j(t)) = M_1(f_j^{in}) - \int_0^t \int_0^j \int_{j-y}^j x K(x,y) f_j(s,x) f_j(s,y)\ \mathrm{d}y\mathrm{d}x\mathrm{d}s\ , \qquad t\ge 0\ , \label{b6}
\end{equation}
and we infer from \eqref{b3}, \eqref{b6}, and the non-negativity of $K$ and $f_j$ that 
\begin{equation}
M_1(f_j(t)) \le \varrho = M_1(f^{in})\ , \qquad t\ge 0\ . \label{b7} 
\end{equation}

\subsection{Moment Estimates} \label{sec2.3}

Let us first draw a couple of consequences of \eqref{aaCFC}.

\begin{lemma}\label{Lemb1}
Let $K$, $a$, and $b$ be coagulation and fragmentation coefficients satisfying \eqref{aaCFC}. Then
\begin{equation}
\frac{1}{2} \le \alpha \le \frac{\lambda}{2} \le \lambda-\alpha \le 1\ , \label{b8}
\end{equation}
\begin{equation}
2K_0 (xy)^{\lambda/2} \le K(x,y) \le K_0 \sqrt{xy} \left( x^{\lambda-1} + y^{\lambda-1} \right)\ , \qquad (x,y)\in (0,\infty)^2\ . \label{b9}
\end{equation}
\end{lemma}

\begin{proof}
The first assertion \eqref{b8} readily follows from \eqref{aa1}. Consider next $(x,y)\in (0,\infty)^2$. On the one hand, it follows from the  Cauchy-Schwarz inequality that 
\begin{equation*}
(xy)^{\lambda/2} = x^{\alpha/2} y^{(\lambda-\alpha)/2} x^{(\lambda-\alpha)/2} y^{\alpha/2} \le \frac{1}{2} \left( x^\alpha y^{\lambda-\alpha} + x^{\lambda-\alpha} y^\alpha \right) = \frac{K(x,y)}{2K_0}\ .
\end{equation*}
On the other hand, using Young's inequality gives
\begin{align*}
K(x,y) & = K_0 \sqrt{xy} \left( x^{(2\alpha-1)/2} y^{(2\lambda-2\alpha-1)/2} +  x^{(2\lambda-2\alpha-1)/2} y^{(2\alpha-1)/2} \right) \\
& \le K_0 \sqrt{xy} \left( \frac{2\alpha-1}{2(\lambda-1)} x^{\lambda-1} + \frac{2\lambda-2\alpha-1}{2(\lambda-1)} y^{\lambda-1} + \frac{2\lambda-2\alpha-1}{2(\lambda-1)} x^{\lambda-1} + \frac{2\alpha-1}{2(\lambda-1)} y^{\lambda-1} \right) \\
& = K_0 \sqrt{xy} \left( x^{\lambda-1} + y^{\lambda-1} \right)\ ,
\end{align*}
and the proof of \eqref{b9} is complete. 
\end{proof}

The next result is at the heart of the control of solutions to \eqref{a1} for large sizes.

\begin{lemma}\label{Lemb2}
For $(x,y)\in (0,\infty)^2$, 
$$
(x+y) \ln(x+y) \le x \ln x + y \ln y + 2\ln 2 \sqrt{xy}\ .
$$
\end{lemma}

\begin{proof}
Fix $y>0$ and define 
$$
J(x) := (x+y) \ln(x+y) - x \ln x - y \ln y - 2\ln 2 \sqrt{xy}\ , \qquad x>0\ .
$$
Then $J'(x)=J_1(y/x)$ for $x>0$ with $J_1(r):= \ln(1+r) - \ln 2 \sqrt{r}$ for $r>0$. Then $J_1(1)=0$ and studying the variation of $J_1$ reveals that there is $r_0>1$ such that $J_1<0$ in $(0,1)\cup (r_0,\infty)$ and $J_1>0$ in $(1,r_0)$. Consequently, $J$ reaches its maximum twice, at $x=0$ and $x=y$, from which the non-positivity of $J$ on $(0,\infty)$ follows.
\end{proof}

After this preparation, we are in a position to state and prove the first estimates for $f_j$ for small and large sizes.

\begin{lemma}\label{Lemb3}
\refstepcounter{NumConst}\label{cst1}
Let $m\in [m_1,1)$ and set $\delta_\varrho := K_0 \ln{2} (\varrho_\star-\varrho)>0$. There is $C_{\ref{cst1}}(m)>0$ depending on $m$ such that, for $t\ge 0$, 
\begin{align*}
& \int_0^\infty x|\ln{x}| f_j(t,x)\ \mathrm{d}x + \frac{1}{e (1-m)} M_m(f_j(t)) + \delta_\varrho \int_0^t M_\lambda(f_j(s))\ \mathrm{d}s \\
& \qquad + \ln{j} \int_0^t \int_0^j \int_{j-y}^j x K(x,y) f_j(s,x) f_j(s,y)\ \mathrm{d}x\mathrm{d}y\mathrm{d}s \\
& \qquad\qquad  \le \int_0^\infty x|\ln{x}| f^{in}(x)\ \mathrm{d}x + \frac{2}{e (1-m)} M_m(f^{in}) + C_{\ref{cst1}}(m)t\ .
\end{align*}
\end{lemma}

\begin{proof}
Since $m\le 1$, there holds
\begin{equation*}
\chi_m(x,y) := (x+y)^m - x^m -y^m \le 0\ , \qquad (x,y)\in (0,\infty)^2\ ,
\end{equation*} 
while \eqref{aa4} and \eqref{aa6} give
\begin{equation*}
N_m(y) := y^m - \int_0^y x^m b(x,y)\ \mathrm{d}x = (1 - \mathfrak{b}_{m,1}) y^m \ge - \mathfrak{b}_{m,1} y^m\ , \qquad y\in (0,\infty)\ .
\end{equation*}
Consequently, we infer from \eqref{aa2}, \eqref{b5} (with $\vartheta(x) = x^m$, $x\in (0,2j)$), and the non-negativity of $f_j$ and $K$ that, for $t\ge 0$, 
\begin{equation}
\frac{\mathrm{d}}{\mathrm{d}t} M_m(f_j(t)) \le a_0 \mathfrak{b}_{m,1} M_{m+\lambda-1}(f_j(t))\ . \label{b10}
\end{equation}
As the choice of $m$ ensures that $m+\lambda-1\in [1,\lambda]$, it follows from \eqref{b7} and H\"older's inequality that 
\begin{align*}
M_{m+\lambda-1}(f_j(t)) & \le M_\lambda(f_j(t))^{(m+\lambda-2)/(\lambda-1)} M_1(f_j(t))^{(1-m)/(\lambda-1)} \\
& \le \varrho^{(1-m)/(\lambda-1)}  M_\lambda(f_j(t))^{(m+\lambda-2)/(\lambda-1)}\ ,
\end{align*}
which gives, together with \eqref{b10} and Young's inequality,
\begin{equation}
\frac{\mathrm{d}}{\mathrm{d}t} M_m(f_j(t)) \le \frac{e (1-m)\delta_\varrho}{3} M_\lambda(f_j(t)) + \frac{e(1-m) C_{\ref{cst1}}(m)}{3}\ , \qquad t\ge 0\ . \label{b11}
\end{equation}

We next set $\bar{\vartheta}(x) := x\ln{x}$, $x\in (0,\infty)$, and notice that 
\begin{equation*}
K(x,y) \chi_{\bar{\vartheta}}(x,y) \le 2 \ln{2} \sqrt{xy} K(x,y) \le 2 K_0 \ln{2} (x^\lambda y + x y^\lambda)\ , \qquad (x,y)\in (0,\infty)^2\ ,
\end{equation*}
by Lemma~\ref{Lemb1} and Lemma~\ref{Lemb2} and
\begin{equation*}
K(x,y) \bar{\vartheta}(x+y) \ge (x+y) K(x,y) \ln{j} \ge 0\ , \qquad x\in (j-y,j)\ , \ y\in (0,j)\ .
\end{equation*} 
Also, owing to \eqref{aa5} and \eqref{aa9}, 
\begin{equation*}
N_{\bar{\vartheta}}(y) = y \ln{y} - \int_0^1 yz \ln{(yz)} B(z)\ \mathrm{d}z = y \int_0^1 z |\ln{z}| B(z)\ \mathrm{d}z = \mathfrak{b}_{\ln} y\ , \qquad y>0\ .
\end{equation*}
Collecting the previous estimates, it follows from \eqref{aa2}, \eqref{b5} with $\vartheta(x)=\bar{\vartheta}(x)$, $x\in (0,2j)$, and \eqref{b7} that, for $t\ge 0$, 
\begin{align}
\frac{\mathrm{d}}{\mathrm{d}t} \int_0^\infty \bar{\vartheta}(x) f_j(t,x)\ \mathrm{d}x & \le 2K_0 \ln{2} M_1(f_j(t)) M_\lambda(f_j(t)) - a_0 \mathfrak{b}_{\ln} M_\lambda(f_j(t)) \nonumber \\
& \qquad - \frac{\ln{j}}{2} \int_0^j \int_{j-y}^j (x+y) K(x,y) f_j(t,x) f_j(t,y)\ \mathrm{d}x\mathrm{d}y \nonumber \\
& \le 2 K_0 \ln{2} (\varrho-\varrho_\star) M_\lambda(f_j(t)) \nonumber \\
& \qquad -\ln{j} \int_0^j \int_{j-y}^j x K(x,y) f_j(t,x) f_j(t,y)\ \mathrm{d}x\mathrm{d}y \nonumber \\
& = -2 \delta_\varrho M_\lambda(f_j(t)) -\ln{j} \int_0^j \int_{j-y}^j x K(x,y) f_j(t,x) f_j(t,y)\ \mathrm{d}x\mathrm{d}y \ . \label{b12}
\end{align}
Combining \eqref{b11} and \eqref{b12} leads us to
\begin{align*}
& \frac{\mathrm{d}}{\mathrm{d}t} \left[ \int_0^\infty x \ln{(x)} f_j(t,x)\ \mathrm{d}x + \frac{3}{e(1-m)} M_m(f_j(t)) \right] \\
& \qquad \le - \delta_\varrho M_\lambda(f_j(t)) -\ln{j} \int_0^j \int_{j-y}^j x K(x,y) f_j(t,x) f_j(t,y)\ \mathrm{d}x\mathrm{d}y + C_{\ref{cst1}}(m)\ .
\end{align*}
Hence, after integration with respect to time,
\begin{align*}
& \int_0^\infty x \ln{(x)} f_j(t,x)\ \mathrm{d}x + \frac{3}{e(1-m)} M_m(f_j(t)) + \delta_\varrho \int_0^t M_\lambda(f_j(s))\ \mathrm{d}s \\ 
& \qquad + \ln{j} \int_0^t \int_0^j \int_{j-y}^j x K(x,y) f_j(s,x) f_j(s,y)\ \mathrm{d}x\mathrm{d}y\mathrm{d}s \\
& \qquad\qquad \le \int_0^\infty x \ln{(x)} f_j^{in}(x)\ \mathrm{d}x + \frac{3}{e(1-m)} M_m(f_j^{in}) + C_{\ref{cst1}}(m)t \\
& \qquad\qquad \le \int_0^\infty x |\ln{x}| f^{in}(x)\ \mathrm{d}x + \frac{3}{e(1-m)} M_m(f^{in}) + C_{\ref{cst1}}(m)t
\end{align*}
for $t\ge 0$. We finally use the inequality
\begin{equation}
x |\ln{x}| - \frac{2 x^m}{e(1-m)} \le x \ln{x} \le x |\ln{x}| \ , \qquad x>0\ , \label{z1}
\end{equation}
to complete the proof.
\end{proof}

We next turn to estimates for moments of order $m\in (-\nu-1,m_1)$.

\begin{lemma}\label{Lemb4}
\refstepcounter{NumConst}\label{cst2}
Consider $m\in (-\nu-1,m_1)$ and assume additionally that $f^{in}\in X_m$. There is $C_{\ref{cst2}}(m)>0$ depending on $m$ such that
\begin{equation*}
M_m(f_j(t)) \le \max\{ M_m(f^{in}) , C_{\ref{cst2}}(m) \} (2+t)^{(\lambda-m)/(\lambda-1)}\ , \qquad t\ge 0\ .
\end{equation*} 
\end{lemma}

\begin{proof}
Since $m\le 1$, we argue as at the beginning of the proof of Lemma~\ref{Lemb3} to derive from \eqref{b5} that
\begin{equation*}
\frac{\mathrm{d}}{\mathrm{d}t} M_m(f_j(t)) \le a_0 \mathfrak{b}_{m,1} M_{m+\lambda-1}(f_j(t))\ , \qquad t\ge 0\ .
\end{equation*}
The range of $m$ and \eqref{aa1} next imply that $m+\lambda-1\in (m,\lambda)$ and we infer from H\"older's inequality that
\begin{equation*}
M_{m+\lambda-1}(f_j(t)) \le M_\lambda(f_j(t))^{(\lambda-1)/(\lambda-m)} M_m(f_j(t))^{(1-m)/(\lambda-m)}\ , \qquad t\ge 0\ .
\end{equation*}
Combining the previous inequalities leads us to
\begin{equation*}
\frac{\mathrm{d}}{\mathrm{d}t} M_m(f_j(t)) \le a_0 \mathfrak{b}_{m,1} M_\lambda(f_j(t))^{(\lambda-1)/(\lambda-m)} M_m(f_j(t))^{(1-m)/(\lambda-m)}\ , 
\end{equation*}
hence, after integration with respect to time,
\begin{equation*}
M_m(f_j(t))^{(\lambda-1)/(\lambda-m)} \le M_m(f_j^{in})^{(\lambda-1)/(\lambda-m)} + \frac{\lambda-1}{\lambda-m} a_0 \mathfrak{b}_{m,1} \int_0^t M_\lambda(f_j(s))^{(\lambda-1)/(\lambda-m)} \mathrm{d}s 
\end{equation*}
for $t\ge 0$. It then follows from \eqref{b00}, \eqref{b3}, Lemma~\ref{Lemb3}, and H\"older's inequality that
\begin{align*}
M_m(f_j(t))^{(\lambda-1)/(\lambda-m)} & \le M_m(f^{in})^{(\lambda-1)/(\lambda-m)} \\
& \qquad + a_0 \mathfrak{b}_{m,1} t^{(1-m)/(\lambda-m)} \left( \int_0^t M_\lambda(f_j(s)) \mathrm{d}s \right)^{(\lambda-1)/(\lambda-m)} \\
& \le M_m(f^{in})^{(\lambda-1)/(\lambda-m)} \\
& \qquad + a_0 \mathfrak{b}_{m,1} t^{(1-m)/(\lambda-m)} \left( \frac{\sigma +C_{\ref{cst1}}(m_1) t}{\delta_\varrho} \right)^{(\lambda-1)/(\lambda-m)} \\
& \le \max\left\{ M_m(f^{in}) , C_{\ref{cst2}}(m)  \right\}^{(\lambda-1)/(\lambda-m)} (2+t)\ ,
\end{align*}
where
\begin{equation*}
C_{\ref{cst2}}(m) := \frac{\sigma +C_{\ref{cst1}}(m_1)}{\delta_\varrho} \left( a_0 \mathfrak{b}_{m,1} \right)^{(\lambda-m)/(\lambda-1)} \ .
\end{equation*}
Lemma~\ref{Lemb4} then readily follows.
\end{proof}

We end up the study of the evolution of moments with moments of higher order. 

\begin{lemma}\label{Lemb5}
\refstepcounter{NumConst}\label{cst3}
Let $m> 1+\lambda-\alpha$ and assume additionally that $f^{in}\in X_m$. For every $T>0$, there is $C_{\ref{cst3}}(m,T)>0$ depending on $m$ and $T$ such that
\begin{equation*}
M_m(f_j(t)) \le \max\left\{ M_m(f^{in}) , C_{\ref{cst3}}(m,T) \right\}\ , \qquad t\in [0,T]\ .
\end{equation*}
\end{lemma}

\begin{proof} We first recall that, since $m>1$, there is $c_m>0$ depending only on $m$ such that
\begin{equation*}
\chi_m(x,y) = (x+y)^m - x^m - y^m \le c_m \left( x^{m-1} y + x y^{m-1} \right)\ , \qquad (x,y)\in (0,\infty)^2\ ,
\end{equation*}
see \cite{BLLxx} or \cite[p.~216]{ClKa99} for instance. Let $t\ge 0$ and $R>1$ and define
\begin{equation*}
Q_j(t,R) := \int_R^\infty x f_j(t,x)\ \mathrm{d}x\ .
\end{equation*}
A straightforward consequence of \eqref{b00} and Lemma~\ref{Lemb3} is that
\begin{equation}
Q_j(t,R) \le \frac{\sigma + C_{\ref{cst1}}(m_1)t}{\ln{R}}\ . \label{b13}
\end{equation} 
It follows from \eqref{aa3} and the previous inequality that
\begin{align*}
P_j(t) & := \frac{1}{2} \int_0^j \int_0^j K(x,y) \chi_m(x,y) f_j(t,x) f_j(t,y)\ \mathrm{d}y\mathrm{d}x \\
& \le \frac{c_m}{2} \int_0^j \int_0^j K(x,y) \left( x^{m-1} y + x y^{m-1} \right) f_j(t,x) f_j(t,y)\ \mathrm{d}y\mathrm{d}x \\
& \le  K_0 c_m \left[ M_{m-1+\alpha}(f_j(t)) M_{1+\lambda-\alpha}(f_j(t)) + M_{m+\lambda-1-\alpha}(f_j(t)) M_{1+\alpha}(f_j(t)) \right]\ .
\end{align*} 
On the one hand, owing to \eqref{b8} and the range of $m$, both $m+\lambda-\alpha-1$ and $m-1+\alpha$ belong to $[1,m]$ and we infer from \eqref{b7} and H\"older's inequality that
\begin{equation*}
M_{m+\lambda-\alpha-1}(f_j(t)) \le \varrho^{(\alpha+1-\lambda)/(m-1)} M_m(f_j(t))^{(m+\lambda-\alpha-2)/(m-1)} \ ,
\end{equation*}
and
\begin{equation*}
M_{m-1+\alpha}(f_j(t)) \le \varrho^{(1-\alpha)/(m-1)} M_m(f_j(t))^{(m+\alpha-2)/(m-1)} \ .
\end{equation*}
On the other hand, since $1+\alpha\in (1,m)$ and $1+\lambda-\alpha\in (1,m)$ due to \eqref{b8} and the choice of $m$, it follows from \eqref{b7}, \eqref{b8}, and H\"older's inequality
\begin{align*}
M_{1+\alpha}(f_j(t)) & \le R^\alpha \int_0^R x f_j(t,x)\ \mathrm{d}x \\
& \qquad + \left( \int_R^\infty x^m f_j(t,x)\ \mathrm{d}x \right)^{\alpha/(m-1)} Q_j(t,R)^{(m-1-\alpha)/(m-1)} \\
& \le \varrho R + Q_j(t,R)^{(m-1-\alpha)/(m-1)} M_m(f_j(t))^{\alpha/(m-1)} \\
& \le \varrho R + \varrho^{(\lambda-2\alpha)/(m-1)} Q_j(t,R)^{(m-1-\lambda+\alpha)/(m-1)} M_m(f_j(t))^{\alpha/(m-1)}\ ,
\end{align*}
and
\begin{align*}
M_{1+\lambda-\alpha}(f_j(t)) & \le R^{\lambda-\alpha} \int_0^R x f_j(t,x)\ \mathrm{d}x \\
& \qquad + \left( \int_R^\infty x^m f_j(t,x)\ \mathrm{d}x \right)^{(\lambda-\alpha)/(m-1)} Q_j(t,R)^{(m-1-\lambda+\alpha)/(m-1)} \\
& \le \varrho R + Q_j(t,R)^{(m-1-\lambda+\alpha)/(m-1)} M_m(f_j(t))^{(\lambda-\alpha)/(m-1)}\ .
\end{align*}
Collecting the above estimates leads us to
\begin{align}
P_j(t) & \le C(m) R \left[ M_m(f_j(t))^{(m+\alpha-2)/(m-1)} + M_m(f_j(t))^{(m+\lambda-\alpha-2)/(m-1)} \right] \nonumber \\
& \qquad + C(m) Q_j(t,R)^{(m-1-\lambda+\alpha)/(m-1)} M_m(f_j(t))^{(m+\lambda-2)/(m-1)}\ . \label{b14}
\end{align}

Now, since
\begin{equation*}
N_m(y) = y^m - \int_0^y x^m b(x,y)\ \mathrm{d}x = (1 - \mathfrak{b}_{m,1}) y^m\ , \qquad y\ge 0\ ,
\end{equation*}
and $1 - \mathfrak{b}_{m,1}>0$ by \eqref{aa4}, \eqref{aa5}, and \eqref{aa6}, it follows from \eqref{aa2}, \eqref{b5}, and \eqref{b14} that
\refstepcounter{NumConst}\label{cst4}
\begin{align*}
\frac{\mathrm{d}}{\mathrm{d}t} M_m(f_j(t)) & \le C_{\ref{cst4}}(m) R \left[ M_m(f_j(t))^{(m+\alpha-2)/(m-1)} + M_m(f_j(t))^{(m+\lambda-\alpha-2)/(m-1)} \right] \\
& \qquad + C_{\ref{cst4}}(m) Q_j(t,R)^{(m-1-\lambda+\alpha)/(m-1)} M_m(f_j(t))^{(m+\lambda-2)/(m-1)} \\
& \qquad - a_0 (1-\mathfrak{b}_{m,1}) M_{m+\lambda-1}(f_j(t))\ .
\end{align*}
By \eqref{b7} and H\"older's inequality, 
\begin{equation*}
M_m(f_j(t)) \le \varrho^{(\lambda-1)/(m+\lambda-2)} M_{m+\lambda-1}(f_j(t))^{(m-1)/(m+\lambda-2)}\ ,
\end{equation*}
so that \refstepcounter{NumConst}\label{cst5}
\begin{align}
\frac{\mathrm{d}}{\mathrm{d}t} M_m(f_j(t)) & \le C_{\ref{cst4}}(m) R \left[ M_m(f_j(t))^{(m+\alpha-2)/(m-1)} + M_m(f_j(t))^{(m+\lambda-\alpha-2)/(m-1)} \right] \nonumber \\
& \qquad + C_{\ref{cst4}}(m) Q_j(t,R)^{(m-1-\lambda+\alpha)/(m-1)} M_m(f_j(t))^{(m+\lambda-2)/(m-1)} \label{b15}\\
& \qquad - 4 C_{\ref{cst5}}(m) M_m(f_j(t))^{(m+\lambda-2)/(m-1)}\ , \nonumber
\end{align}
where
\begin{equation*}
C_{\ref{cst5}}(m) := \frac{a_0 (1-\mathfrak{b}_{m,1}) \varrho^{-(\lambda-1)/(m-1)}}{4}\ . 
\end{equation*}
Observe that $m+\alpha-2<m+\lambda-2$ and $m+\lambda-\alpha-2<m+\lambda-2$, so that the first term on the right-hand side of \eqref{b15} is strictly dominated by the last one.

Now, let $T>0$ and consider $t\in [0,T]$. It readily follows from \eqref{b13} that there is $R_T>1$ large enough such that 
\begin{equation*}
Q_j(t,R_T) \le \frac{\sigma + C_{\ref{cst1}}(m_1) T}{\ln{R_T}} \le \left( \frac{2 C_{\ref{cst5}}(m)}{C_{\ref{cst4}}(m)} \right)^{(m-1)/(m-1-\lambda+\alpha)}\ , \qquad t\in [0,T]\ .
\end{equation*}
Taking $R=R_T$ in \eqref{b15} and using the above inequality as well as Young's inequality, we end up with
\refstepcounter{NumConst}\label{cst6}
\begin{equation*}
\frac{\mathrm{d}}{\mathrm{d}t} M_m(f_j(t)) \le C_{\ref{cst6}}(m,T) - C_{\ref{cst5}}(m) M_m(f_j(t))^{(m+\lambda-2)/(m-1)}\ , \qquad t\in [0,T]\ .
\end{equation*}
We then infer from the comparison principle that
\begin{equation*}
M_m(f_j(t)) \le \max\left\{ M_m(f_j^{in}) , \left( \frac{C_{\ref{cst6}}(m,T)}{C_{\ref{cst5}}(m)} \right)^{(m-1)/(m+\lambda-2)} \right\}
\end{equation*}
for $t\in [0,T]$, and we use the obvious bound $M_m(f_j^{in})\le M_m(f^{in})$ due to \eqref{b3} to complete the proof.
\end{proof}

\subsection{Uniform Integrability} \label{sec2.4}

The outcome of the previous section guarantees that there is no leak of matter for small and large sizes. The purpose of the next result is to prevent concentration of matter at a finite size by showing the uniform integrability of the sequence $(f_j)_j$ in $X_{m_1}$. To this end, we first recall that, since $f^{in} \in X_{m_1}\subset X_{m_0}\cap X_1$, a refined version of the de la Vall\'ee-Poussin theorem \cite{BLLxx, Le77} ensures that there is a non-negative and convex function $\Phi\in C^1([0,\infty))$ with concave first derivative $\Phi'$ such that
\begin{subequations}\label{b16}
\begin{align}
& \Phi(0)=\Phi'(0)=0\ , \qquad \lim_{r\to\infty} \Phi'(r) = \lim_{r\to\infty} \frac{\Phi(r)}{r} = \infty\ , \label{b16a} \\
& \mathcal{I}^{in} := \int_0^\infty x^{m_1} \Phi(f^{in}(x))\ \mathrm{d}x < \infty\ . \label{b16b}
\end{align}
\end{subequations}
As a consequence of \eqref{b16}, the convexity of $\Phi$, and the concavity of $\Phi'$, there holds
\begin{equation}
\Phi(r) \ge \Phi_1(r) := r \Phi'(r) - \Phi(r) \ge 0\ , \qquad r\in [0,\infty)\ , \label{b17a}
\end{equation}
and
\begin{equation}
s \Phi'(r) \le r \Phi'(r) + \Phi(s) - \Phi(r) = \Phi(s) + \Phi_1(r)\ , \qquad (r,s)\in [0,\infty)^2\ . \label{b17b}
\end{equation}
It is furthermore possible to construct $\Phi$ such that
\begin{equation}
S_p := \sup_{r\ge 0}\left\{ \frac{\Phi(r)}{r^p} \right\} < \infty \;\text{ for all }\; p\in (1,2]\ , \label{b18}
\end{equation} 
see \cite{BLLxx}. We next fix two additional parameters $\delta_1\in [0,1)$ and $p_1\in (1,2)$ which only depend on $\lambda$ and $\nu$ and satisfy
\begin{equation}
\delta_1 := \frac{2-\lambda+(m_1+1-\lambda)_+}{2} \;\text{ and }\; 1 < p_1 < 1 + \min\left\{ \frac{m_1+\nu+1}{\delta_1-\nu} , \frac{\lambda-1}{1+\delta_1} \right\}\ . \label{b19}
\end{equation}

\begin{lemma}\label{Lemb6}
\refstepcounter{NumConst}\label{cst7}
Let $T>0$. There is $C_{\ref{cst7}}(T,S_{p_1})>0$ depending on $T$ and $S_{p_1}$ such that
\begin{equation*}
\int_0^\infty x^{m_1} \Phi(f_j(t,x))\ \mathrm{d}x \le C_{\ref{cst7}}(T,S_{p_1}) \left( 1 + \mathcal{I}^{in} \right)\ , \qquad t\in [0,T]\ .
\end{equation*}
\end{lemma}

In the proof of Lemma~\ref{Lemb6}, we first exploit the monotonicity of the coagulation kernel, the subadditivity of the weight function $x\mapsto x^{m_1}$, and the specific choice \eqref{b1} of the truncation of $K$ to show that the contribution of the coagulation term is non-positive, a property which has been uncovered in \cite{Buro83} and subsequently used in \cite{BLLxx, Dubo94b, LaMi02c, MiRR03}. The analysis of the contribution of the fragmentation term is more delicate and the specific choice of the additional parameters $\delta_1$ and $p_1$ come into play there.

\begin{proof}
Let $t\in [0,T]$. On the one hand, it follows from Fubini's theorem, the symmetry of $K$, and the subadditivity of $x\mapsto x^{m_1}$ that
\begin{align*}
X_j(t) & := \int_0^j x^{m_1} \Phi'(f_j(t,x)) \mathcal{C}_j f_j(t,x)\ \mathrm{d}x \\
& = \frac{1}{2} \int_0^j \int_0^{j-y} (x+y)^{m_1} K(x,y) \Phi'(f_j(t,x+y)) f_j(t,y) f_j(t,x)\ \mathrm{d}x\mathrm{d}y \\
& \qquad - \int_0^j \int_0^j x^{m_1} K(x,y) \Phi'(f_j(t,x)) f_j(t,y) f_j(t,x)\ \mathrm{d}x\mathrm{d}y \\
& \le \frac{1}{2} \int_0^j \int_0^{j-y} \left( x^{m_1} + y^{m_1} \right) K(x,y) \Phi'(f_j(t,x+y)) f_j(t,y) f_j(t,x)\ \mathrm{d}x\mathrm{d}y \\
& \qquad - \int_0^j \int_0^j x^{m_1} K(x,y) \Phi'(f_j(t,x)) f_j(t,y) f_j(t,x)\ \mathrm{d}x\mathrm{d}y \\
& \le \int_0^j \int_0^{j-y} x^{m_1} K(x,y) \Phi'(f_j(t,x+y)) f_j(t,y) f_j(t,x)\ \mathrm{d}x\mathrm{d}y \\
& \qquad - \int_0^j \int_0^j x^{m_1} K(x,y) \Phi'(f_j(t,x)) f_j(t,y) f_j(t,x)\ \mathrm{d}x\mathrm{d}y \ .
\end{align*}
Owing to \eqref{b17b} (with $r=f_j(t,x+y)$ and $s=f_j(t,x)$), we further obtain
\begin{align*}
X_j(t) & \le \int_0^j \int_0^{j-y} x^{m_1} K(x,y) \Phi_1(f_j(t,x+y)) f_j(t,y)\ \mathrm{d}x\mathrm{d}y \\
& \qquad + \int_0^j \int_0^{j-y} x^{m_1} K(x,y) \Phi(f_j(t,x)) f_j(t,y)\ \mathrm{d}x\mathrm{d}y \\
& \qquad - \int_0^j \int_0^j x^{m_1} K(x,y) \Phi'(f_j(t,x)) f_j(t,y) f_j(t,x)\ \mathrm{d}x\mathrm{d}y \\
& \le \int_0^j \int_y^j (x-y)^{m_1} K(x-y,y) \Phi_1(f_j(t,x)) f_j(t,y)\ \mathrm{d}x\mathrm{d}y \\
& \qquad - \int_0^j \int_0^j x^{m_1} K(x,y) \Phi_1(f_j(t,x)) f_j(t,y)\ \mathrm{d}x\mathrm{d}y \\
& \le \int_0^j \int_0^j \left[ (x-y)^{m_1} K(x-y,y) \mathbf{1}_{(0,x)}(y) - x^{m_1} K(x,y) \right] \Phi_1(f_j(t,x)) f_j(t,y)\ \mathrm{d}x\mathrm{d}y\ .
\end{align*}
We finally infer from the monotonicity of $x\mapsto x^{m_1} K(x,y)$ and the non-negativity \eqref{b17a} of $\Phi_1$ that 
\begin{equation}
X_j(t) = \int_0^j x^{m_1} \Phi'(f_j(t,x)) \mathcal{C}_j f_j(t,x)\ \mathrm{d}x \le 0\ . \label{b20}
\end{equation}

On the other hand, using again Fubini's theorem and \eqref{b17b} (with $r=f_j(t,x)$ and $s=b(x,y)x^{-\delta_1}$),
\begin{align*}
Y_j(t) & := \int_0^j x^{m_1} \Phi'(f_j(t,x)) \mathcal{F}_j f_j(t,x)\ \mathrm{d}x \\
& \le \int_0^j a(y) f_j(t,y) \int_0^y x^{m_1+\delta_1} \frac{b(x,y)}{x^{\delta_1}} \Phi'(f_j(t,x))\ \mathrm{d}x\mathrm{d}y \\
& \le \int_0^j a(y) f_j(t,y) \int_0^y x^{m_1+\delta_1} \Phi_1(f_j(t,x))\ \mathrm{d}x\mathrm{d}y + Z_j(t)\ ,
\end{align*}
where 
\begin{equation*}
Z_j(t) := \int_0^j a(y) f_j(t,y) \int_0^y x^{m_1+\delta_1} \Phi\left( b(x,y) x^{-\delta_1} \right)\ \mathrm{d}x\mathrm{d}y \ .
\end{equation*}
Since $p_1\in (1,2)$, we infer from \eqref{aa2}, \eqref{aa4}, and \eqref{b18} that
\begin{align*}
Z_j(t) & \le S_{p_1} \int_0^j a(y) y^{-p_1} f_j(t,y) \int_0^y x^{m_1+\delta_1(1-p_1)} \left[ B\left( \frac{x}{y} \right) \right]^{p_1}\ \mathrm{d}x\mathrm{d}y \\
& \le S_{p_1} \mathfrak{b}_{m_1+\delta_1(1-p_1),p_1} \int_0^y a(y) y^{m_1+(1+\delta_1)(1-p_1)} f_j(t,y)\ \mathrm{d}y \\
& \le a_0 S_{p_1} \mathfrak{b}_{m_1+\delta_1(1-p_1),p_1} M_{\lambda-1+m_1+(1+\delta_1)(1-p_1)}(f_j(t))\ .
\end{align*}
At this point, we notice that, due to \eqref{b19} and the non-positivity of $\nu$,
\begin{align*}
m_1+ \delta_1 (1-p_1) & \ge m_1+\delta_1(1-p_1) + p_1 \nu = m_1 + \nu + (\delta_1-\nu)(1-p_1) \\
& > m_1+\nu - (m_1+\nu+1) = -1\ ,
\end{align*}
so that $(m_1+\delta_1(1-p_1),p_1)\in \mathcal{A}_\nu$ and $\mathfrak{b}_{m_1+\delta_1(1-p_1),p_1}$ is finite. Consequently, we deduce from \eqref{aa2}, \eqref{b17a}, and the estimate on $Z_j(t)$ that
\begin{align*}
Y_j(t) & \le a_0 M_{\lambda-1+\delta_1}(f_j(t)) \int_0^j x^{m_1} \Phi(f_j(t,x))\ \mathrm{d}x \\
& \qquad + a_0 S_{p_1} \mathfrak{b}_{m_1+\delta_1(1-p_1),p_1} M_{\lambda-1+m_1+(1+\delta_1)(1-p_1)}(f_j(t))\ .
\end{align*}
Moreover, using once more \eqref{b19}, we realize that
\begin{equation*}
\lambda-1+m_1+(1+\delta_1)(1-p_1) \in [m_1,\lambda] \;\text{ and }\; \lambda-1+\delta_1 \in [m_1,1]\ ,
\end{equation*}
and we infer from \eqref{b7}, Lemma~\ref{Lemb3} (with $m=m_1$), and H\"older's and Young's inequalities that 
\begin{align*}
M_{\lambda-1+m_1+(1+\delta_1)(1-p_1)}(f_j(t)) & \le \frac{\lambda-1+(1+\delta_1)(1-p_1)}{\lambda-m_1} M_\lambda(f_j(t)) \\
& \qquad + \frac{1-m_1 + (1+\delta_1)(p_1-1)}{\lambda-m_1} M_{m_1}(f_j(t)) \\
& \le M_\lambda(f_j(t)) + C(T)\ ,
\end{align*}
and
\begin{equation*}
M_{\lambda-1+\delta_1}(f_j(t)) \le \varrho^{(\lambda-1+\delta_1-m_1)/(1-m_1)} M_{m_1}(f_j(t))^{(2-\lambda-\delta_1)/(1-m_1)} \le C(T)\ . 
\end{equation*}
Consequently, \refstepcounter{NumConst}\label{cst8}
\begin{equation}
Y_j(t) \le C_{\ref{cst8}}(T,S_{p_1}) \left( 1 + M_\lambda(f_j(t)) + \int_0^j x^{m_1} \Phi(f_j(t,x))\ \mathrm{d}x \right)\ . \label{b21}
\end{equation}

It now follows from \eqref{b4a}, \eqref{b20}, and \eqref{b21} that, for $t\in [0,T]$,
\begin{equation*}
\frac{\mathrm{d}}{\mathrm{d}t} \int_0^\infty x^{m_1} \Phi(f_j(t,x))\ \mathrm{d}x \le C_{\ref{cst8}}(T,S_{p_1}) \left( 1 + M_\lambda(f_j(t)) + \int_0^\infty x^{m_1} \Phi(f_j(t,x))\ \mathrm{d}x \right)\ ,
\end{equation*}
hence, after integration with respect to time,
\begin{align*}
\int_0^\infty x^{m_1} \Phi(f_j(t,x))\ \mathrm{d}x & \le e^{C_{\ref{cst8}}(T,S_{p_1}) t} \int_0^\infty x^{m_1} \Phi(f_j^{in})\ \mathrm{d}x \\
& \qquad + C_{\ref{cst8}}(T,S_{p_1}) e^{C_{\ref{cst8}}(T,S_{p_1}) t} \left( t + \int_0^t M_\lambda(f_j(s))\ \mathrm{d}s \right)\ .
\end{align*}
We then use \eqref{b3}, \eqref{b16b}, Lemma~\ref{Lemb3} (with $m=m_1$)), and the monotonicity of $\Phi$ to complete the proof.
\end{proof}

\subsection{Time Equicontinuity} \label{sec2.5}

According to the Dunford-Pettis theorem, the estimates derived in Sections~\ref{sec2.3} and~\ref{sec2.4} along with the superlinearity \eqref{b16a} of $\Phi$ at infinity imply the weak compactness in $X_{m_1}$ of the sequence $(f_j(t))_j$ for all $t\ge 0$. To show the convergence of the sequence $(f_j)_j$, we are left with the compactness with respect to the time variable which we establish now.

\begin{lemma}\label{Lemb7}
\refstepcounter{NumConst}\label{cst9} For $T>0$, there exists $C_{\ref{cst9}}(T)>0$ depending on $T$ such that
\begin{equation*}
\int_0^\infty x^{m_1} |f_j(t_2,x) - f_j(t_1,x)|\ \mathrm{d}x \le C_{\ref{cst9}}(T) \omega_0(t_2-t_1)\ , \qquad 0\le t_1\le t_2\le T\ ,
\end{equation*} 
with $\omega_0(s) := \max\left\{ s^{(1-m_1)/(2-m_1)} , s \right\}$, $s\ge 0$.
\end{lemma}

\begin{proof}
Let $t\in [0,T]$ and $R>1$. On the one hand, we infer from Fubini's theorem that 
\begin{align*}
X_j(t,R) & := \int_0^R x^{m_1} |\mathcal{C}_j f_j(t,x)|\ \mathrm{d}x \\
& \le \frac{1}{2} \int_0^R \int_0^{R-y} (x+y)^{m_1} K(x,y) f_j(t,x) f_j(t,y)\ \mathrm{d}x\mathrm{d}y \\
& \qquad + \int_0^R \int_0^j x^{m_1} K(x,y) f_j(t,x) f_j(t,y)\ \mathrm{d}y\mathrm{d}x\ .
\end{align*}
Using the subadditivity of $x\mapsto x^{m_1}$ and \eqref{aa3}, we further obtain
\begin{align*}
X_j(t,R) & \le \int_0^R \int_0^{R-y} x^{m_1} K(x,y) f_j(t,x) f_j(t,y)\ \mathrm{d}x\mathrm{d}y \\
& \qquad + \int_0^R \int_0^j x^{m_1} K(x,y) f_j(t,x) f_j(t,y)\ \mathrm{d}y\mathrm{d}x \\
& \le 2K_0 \int_0^R \int_0^j x^{m_1} \left( x^\alpha y^{\lambda-\alpha} + x^{\lambda-\alpha} y^\alpha \right) f_j(t,x) f_j(t,y)\ \mathrm{d}y\mathrm{d}x \\
& \le 4 K_0 R^{m_1} M_\alpha(f_j(t)) M_{\lambda-\alpha}(f_j(t))\ .
\end{align*}
Since $m_0\le \alpha \le \lambda-\alpha\le 1$ by \eqref{aa1} and \eqref{aa11}, we deduce from \eqref{b00}, \eqref{b7}, H\"older's inequality, and, either Lemma~\ref{Lemb3} (with $m=m_1$) if $m_1=m_0$, or Lemma~\ref{Lemb4} (with $m=m_0$) if $m_1>m_0$,  that
\begin{equation}
X_j(t,R) \le 4 K_0 R^{m_1} M_1(f_j(t))^{(\lambda-2m_0)/(1-m_0)} M_{m_0}(f_j(t))^{(2-\lambda)/(1-m_0)} \le C(T) R\ . \label{b22}
\end{equation}
On the other hand, using \eqref{aa2}, \eqref{aa4}, \eqref{aa6}, Fubini's theorem, and $(m_1,1)\in \mathcal{A}_\nu$ (see \eqref{aa11b} and Remark~\ref{rema1}), we find 
\begin{align*}
Y_j(t,R) & := \int_0^R x^{m_1} |\mathcal{F}_j f_j(t,x)|\ \mathrm{d}x \\
& \le a_0 \int_0^R x^{m_1+\lambda-1} f_j(t,x)\ \mathrm{d}x + a_0 \mathfrak{b}_{m_1,1} \int_0^\infty y^{m_1+\lambda-1} f_j(t,y)\ \mathrm{d}y \\
& \le a_0 (1+\mathfrak{b}_{m_1,1}) R^{m_1+\lambda-2} M_1(f_j(t)) + a_0 \mathfrak{b}_{m_1,1} R^{m_1-1} \int_R^\infty y^\lambda f_j(t,y)\ \mathrm{d}y\ .
\end{align*}
Hence, thanks to \eqref{aa1}, \eqref{aa11b}, and \eqref{b7}, 
\begin{equation}
Y_j(t,R) \le C \left( R + R^{m_1-1} M_\lambda(f_j(t)) \right)\ . \label{b23}
\end{equation}
We now combine \eqref{b4a}, \eqref{b22}, and \eqref{b23} to conclude that, for $t\in [0,T]$, 
\begin{equation*}
\int_0^R x^{m-1} |\partial_t f_j(t,x)|\ \mathrm{d}x \le X_j(t,R) + Y_j(t,R) \le C(T) \left( R + R^{m_1-1} M_\lambda(f_j(t)) \right)\ .
\end{equation*}
Therefore, for $0\le t_1 \le t_2 \le T$, we deduce from \eqref{b00}, Lemma~\ref{Lemb3} (with $m=m_1$), and Fubini's theorem that
\refstepcounter{NumConst}\label{cst10}
\begin{align}
\int_0^R x^{m_1} |f_j(t_2,x)-f_j(t_1,x)|\ \mathrm{d}x & \le \int_0^R \int_{t_1}^{t_2} x^{m-1} |\partial_t f_j(t,x)|\ \mathrm{d}t \mathrm{d}x \nonumber \\
& \le C(T) \left( R (t_2-t_1) + R^{m_1-1} \int_{t_1}^{t_2} M_\lambda(f_j(s))\ \mathrm{d}s \right) \nonumber \\
& \le C_{\ref{cst10}}(T) \left( R (t_2-t_1) + R^{m_1-1} \right) \label{b24}
\end{align}
. It then follows from \eqref{aa11b}, \eqref{b7}, and \eqref{b24} that, for $0\le t_1 \le t_2 \le T$ and $R>1$,
\begin{align*}
\int_0^\infty x^{m_1} |f_j(t_2,x)-f_j(t_1,x)|\ \mathrm{d}x & \le \int_0^R x^{m_1} |f_j(t_2,x)-f_j(t_1,x)|\ \mathrm{d}x \\
& \qquad + R^{m_1-1} \int_R^\infty x \left[ f_j(t_1,x)+f_j(t_2,x) \right]\ \mathrm{d}x \\
& \le C_{\ref{cst10}}(T) \left( R (t_2-t_1) + R^{m_1-1} \right) + 2 \varrho R^{m_1-1}\ .
\end{align*}
The previous inequality being valid for any $R>1$, we may choose $R=(t_2-t_1)^{-1/(2-m_1)}$ when $t_2-t_1<1$ and $R=2$ otherwise and thereby complete the proof of Lemma~\ref{Lemb7}.
\end{proof}

\begin{corollary}\label{Corb8}
\refstepcounter{NumConst}\label{cst11} For $T>0$, there exists $C_{\ref{cst11}}(T)>0$ depending on $T$ such that
\begin{equation*}
\int_0^\infty x |f_j(t_2,x) - f_j(t_1,x)|\ \mathrm{d}x \le C_{\ref{cst11}}(T) \omega(t_2-t_1)\ , \qquad 0\le t_1\le t_2\le T\ ,
\end{equation*} 
where
\begin{equation*}
\omega(s) := \inf_{R>1}\left\{ R^{1-m_1} \omega_0(s) + \frac{1}{\ln{R}} \right\}\ , \qquad s\ge 0\ .
\end{equation*}
\end{corollary}

\begin{proof}
Consider $0\le t_1\le t_2\le T$ and $R>1$. We infer from \eqref{aa11b}, \eqref{b00}, Lemma~\ref{Lemb3} (with $m=m_1$), and Lemma~\ref{Lemb7} that
\begin{align*}
\int_0^\infty x |f_j(t_2,x) - f_j(t_1,x)|\ \mathrm{d}x & \le R^{1-m_1} \int_0^R x^{m_1} |f_j(t_2,x) - f_j(t_1,x)|\ \mathrm{d}x \\ 
& \qquad + \frac{1}{\ln{R}} \int_R^\infty x |\ln{x}| \left[ f_j(t_1,x)+f_j(t_2,x) \right]\ \mathrm{d}x \\
& \le C_{\ref{cst9}}(T) R^{1-m_1} \omega_0(t_2-t_1) + \frac{C(T)}{\ln{R}}\ ,
\end{align*}
and Corollary~\ref{Corb8} readily follows since the previous inequality is valid for all $R>1$.
\end{proof}

\subsection{Convergence} \label{sec2.6}

We are now in a position to complete the proof of the first statements in Theorem~\ref{ThmP1}.

\begin{proof}[Proof of Theorem~\ref{ThmP1}~(a): Existence]
Let $T>0$. Since $x\mapsto x|\ln{x}|$ and $\Phi$ are superlinear at infinity, it follows from Lemma~\ref{Lemb3} (with $m=m_1$), Lemma~\ref{Lemb6}, and the Dunford-Pettis theorem that there is a weakly compact subset $\mathcal{K}_T$ of $X_{m_1}\cap X_1$ such that $f_j(t)\in \mathcal{K}_T$ for all $t\in [0,T]$ and $j\ge 1$. In addition, the function $\omega$ defined in Corollary~\ref{Corb8} satisfies $\omega(s)\to 0$ as $s\to 0$, so that the sequence $(f_j)_{j}$ is equicontinuous in $X_{m_1}\cap X_1$ at any $t\in [0,T]$ by Lemma~\ref{Lemb7} and Corollary~\ref{Corb8}. According to a variant of the Arzel\`a-Ascoli theorem \cite[Theorem~A.3.1]{Vrab03}, these properties imply the compactness of the sequence $(f_j)_{j}$ in $C([0,T],X_{m_1,w})$ and in $C([0,T],X_{1,w})$, recalling that $X_{m,w}$ denotes the space $X_m$ endowed with its weak topology. A diagonal process then allows us to construct a subsequence of $(f_j)_{j}$ (not relabeled) and a non-negative function $f\in C([0,\infty),X_{m_1,w}\cap X_{1,w})$ such that
\begin{equation}
f_j \longrightarrow f \;\text{ in }\; C([0,T],X_{m_1,w}\cap X_{1,w}) \;\text{ for all }\; T>0\ . \label{b25}
\end{equation}
A first consequence of, either \eqref{b25} if $m_1=m_0$, or \eqref{b25}, Lemma~\ref{Lemb4} (with $m=m_0$), and Fatou's lemma if $m_0<m_1$, is that
\begin{equation}
f\in L^\infty((0,T),X_{m_0})  \;\text{ for all }\; T>0\ . \label{b26}
\end{equation}
Consider next $t>0$. Thanks to \eqref{b00} and Lemma~\ref{Lemb3} (with $m=m_1$), \refstepcounter{NumConst}\label{cst12}
\begin{align}
\int_0^t \int_0^j \int_{j-y}^j x K(x,y) f_j(s,x) f_j(s,y)\ \mathrm{d}x\mathrm{d}y\mathrm{d}s & \le \frac{\sigma+C_{\ref{cst1}}(m_1) t}{\ln{j}}\ , \label{b28} \\
\int_0^t M_\lambda(f_j(s))\ \mathrm{d}s & \le C_{\ref{cst12}}(t) :=\frac{\sigma+C_{\ref{cst1}}(m_1) t}{\delta_\varrho}\ . \label{b29}
\end{align}
On the one hand, we may pass to the limit as $j\to\infty$ in \eqref{b6} with the help of \eqref{b3}, \eqref{b25}, and \eqref{b28} to conclude that
\begin{equation}
M_1(f(t)) = M_1(f^{in})\ , \qquad t\ge 0\ . \label{b27}
\end{equation}
On the other hand, we infer from \eqref{b25}, \eqref{b29}, and Fatou's lemma that
\begin{equation}
\int_0^t M_\lambda(f(s))\ \mathrm{d}s \le C_{\ref{cst12}}(t)\ , \qquad t\ge 0\ . \label{b30}
\end{equation}

We are left with showing that $f$ satisfies the weak formulation \eqref{b0} of \eqref{a1} for all $\vartheta\in \Theta_{m_1}$. Consider thus $\vartheta\in \Theta_{m_1}$ and observe that, sice $\Theta_{m_1}\subset \Theta_{m_0}$, it follows from \eqref{b5} after integration with respect to time that
\begin{align}
\int_0^\infty \vartheta(x) (f_j-f)(t,x)\ \mathrm{d}x & = \frac{1}{2} \int_0^t \int_0^\infty \int_0^\infty K(x,y) \chi_\vartheta(x,y) f_j(s,x) f_j(s,y)\ \mathrm{d}y\mathrm{d}x\mathrm{d}s \nonumber\\
& \qquad - \int_0^t \int_0^\infty a(x) N_\vartheta(x) f_j(s,x)\ \mathrm{d}x\mathrm{d}s \label{b31} \\
& \qquad - \frac{1}{2} \int_0^t \int_0^j \int_{j-y}^j K(x,y) \vartheta(x+y) f_j(s,x) f_j(s,y)\ \mathrm{d}x\mathrm{d}y\mathrm{d}s\ . \nonumber
\end{align}
First, since $\vartheta\in \Theta_{m_1}$, we infer from \eqref{aa11b} and \eqref{b28} that
\begin{align*}
& \frac{1}{2} \int_0^t \int_0^j \int_{j-y}^j K(x,y) \vartheta(x+y) f_j(s,x) f_j(s,y)\ \mathrm{d}x\mathrm{d}y\mathrm{d}s \\
& \qquad \le \frac{\|\vartheta\|_{C^{0,m_1}}}{2} \int_0^t \int_0^j \int_{j-y}^j (x+y)^{m_1} K(x,y) f_j(s,x) f_j(s,y)\ \mathrm{d}x\mathrm{d}y\mathrm{d}s \\
& \qquad \le \frac{\|\vartheta\|_{C^{0,m_1}}}{2 j^{1-m_1}} \int_0^t \int_0^j \int_{j-y}^j (x+y) K(x,y) f_j(s,x) f_j(s,y)\ \mathrm{d}x\mathrm{d}y\mathrm{d}s \\
& \qquad \le \frac{\|\vartheta\|_{C^{0,m_1}}}{j^{1-m_1}} \int_0^t \int_0^j \int_{j-y}^j x K(x,y) f_j(s,x) f_j(s,y)\ \mathrm{d}x\mathrm{d}y\mathrm{d}s \\
& \qquad \le \frac{\|\vartheta\|_{C^{0,m_1}}}{j^{1-m_1}} \frac{\sigma+C_{\ref{cst1}}(m_1) t}{\ln{j}}\ .
\end{align*}
Therefore,
\begin{equation}
\lim_{j\to\infty} \frac{1}{2} \int_0^t \int_0^j \int_{j-y}^j K(x,y) \vartheta(x+y) f_j(s,x) f_j(s,y)\ \mathrm{d}x\mathrm{d}y\mathrm{d}s = 0\ . \label{b32}
\end{equation}
Next, since $\vartheta\in \Theta_{m_1}$ and $(m_1,1)\in\mathcal{A}_\nu$ by \eqref{aa11b} and Remark~\ref{rema1}, it follows from \eqref{aa4} and \eqref{aa6} that, for $x\in (0,\infty)$,
\begin{align}
\left| N_\vartheta(x) \right| & \le |\vartheta(x)| + \int_0^x |\vartheta(y)| b(y,x)\ \mathrm{d}y \le \|\vartheta\|_{C^{0,m_1}} \left( x^{m_1} + \int_0^x y^{m_1} b(y,x)\ \mathrm{d}y \right) \nonumber \\
& \le (1+\mathfrak{b}_{m_1,1})   \|\vartheta\|_{C^{0,m_1}} x^{m_1}\ . \label{b33}
\end{align}
Then, using \eqref{aa2}, \eqref{b29}, \eqref{b30}, and \eqref{b33}, we obtain, for $R>1$,
\begin{align*}
& \left| \int_0^t \int_0^\infty a(x) N_\vartheta(x) (f_j-f)(s,x)\ \mathrm{d}x\mathrm{d}s \right| \\
& \qquad \le \left| \int_0^t \int_0^R a(x) N_\vartheta(x) (f_j-f)(s,x)\ \mathrm{d}x\mathrm{d}s \right| + \int_0^t \int_R^\infty a(x) N_\vartheta(x) (f_j+f)(s,x)\ \mathrm{d}x\mathrm{d}s \\
& \qquad \le \left| \int_0^t \int_0^R a(x) N_\vartheta(x) (f_j-f)(s,x)\ \mathrm{d}x\mathrm{d}s \right| \\
& \qquad\quad + a_0 (1+\mathfrak{b}_{m_1,1})   \|\vartheta\|_{C^{0,m_1}} \int_0^t \int_R^\infty x^{m_1+\lambda-1} (f_j+f)(s,x)\ \mathrm{d}x\mathrm{d}s \\
& \qquad \le \left| \int_0^t \int_0^R a(x) N_\vartheta(x) (f_j-f)(s,x)\ \mathrm{d}x\mathrm{d}s \right| \\
& \qquad\quad + a_0 (1+\mathfrak{b}_{m_1,1}) \|\vartheta\|_{C^{0,m_1}} R^{m_1-1} \int_0^t \left( M_\lambda(f_j)(s) + M_\lambda(f)(s) \right)\ \mathrm{d}s \\
& \qquad \le \left| \int_0^t \int_0^R a(x) N_\vartheta(x) (f_j-f)(s,x)\ \mathrm{d}x\mathrm{d}s \right| + 2 a_0 (1+\mathfrak{b}_{m_1,1}) \|\vartheta\|_{C^{0,m_1}} C_{\ref{cst12}}(t) R^{m_1-1}\ .
\end{align*}
Now, $x\mapsto a(x) N_\vartheta(x) \mathbf{1}_{(0,R)}(x) x^{-m_1}$ belongs to $L^\infty((0,\infty))$ by \eqref{aa2} and \eqref{b33} which, together with \eqref{b25}, implies that the first term in the right-hand side of the previous inequality converges to zero as $j\to \infty$. Consequently, for all $R>1$, 
\begin{equation*}
\limsup_{j\to\infty} \left| \int_0^t \int_0^\infty a(x) N_\vartheta(x) (f_j-f)(s,x)\ \mathrm{d}x\mathrm{d}s \right| \le 2 a_0 (1+\mathfrak{b}_{m_1,1}) \|\vartheta\|_{C^{0,m_1}} C_{\ref{cst12}}(t) R^{m_1-1}\ .
\end{equation*}
Letting $R\to\infty$ in the above inequality gives, since $m_1<1$,
\begin{equation}
\lim_{j\to\infty} \int_0^t \int_0^\infty a(x) N_\vartheta(x) f_j(s,x)\ \mathrm{d}x\mathrm{d}s = \int_0^t \int_0^\infty a(x) N_\vartheta(x) f(s,x)\ \mathrm{d}x\mathrm{d}s\ . \label{b34}
\end{equation}
Finally, since $m_0<\alpha\le\lambda-\alpha\le 1$ by \eqref{aa1} and \eqref{aa11}, the by now classical argument designed in \cite{Stew89}, and further developed in \cite{BLLxx, ELMP03, LaMi02b}, can be applied to deduce from Lemma~\ref{Lemb4} (with $m=m_0$), \eqref{b25}, and \eqref{b26} that
\begin{align}
& \lim_{j\to\infty} \frac{1}{2} \int_0^t \int_0^\infty \int_0^\infty K(x,y) \chi_\vartheta(x,y) f_j(s,x) f_j(s,y)\ \mathrm{d}y\mathrm{d}x\mathrm{d}s \nonumber\\ 
& \qquad = \frac{1}{2} \int_0^t \int_0^\infty \int_0^\infty K(x,y) \chi_\vartheta(x,y) f(s,x) f(s,y)\ \mathrm{d}y\mathrm{d}x\mathrm{d}s\ . \label{b35}
\end{align}
Collecting \eqref{b32}, \eqref{b34}, and \eqref{b35}, we may take the limit $j\to\infty$ in \eqref{b31} to obtain \eqref{b0}, after handling  the left-hand side of \eqref{b31} with \eqref{b3} and \eqref{b25}.
\end{proof}

\begin{proof}[Proof of Theorem~\ref{ThmP1}~(b): Local boundedness of moments]
Consider $m\in (-\nu-1,m_0)\cup (1+\lambda-\alpha,\infty)$ such that $f^{in}\in X_m$ and let $f$ be the mass-conserving weak solution to \eqref{a1} on $[0,\infty) $ given by \eqref{b25}. For $T>0$, it readily follows from \eqref{b25}, Lemma~\ref{Lemb4} (if $m\in (-\nu-1,m_0)$) or Lemma~\ref{Lemb5} (if $m>1+\lambda-\alpha$), and Fatou's lemma that $f\in L^\infty((0,T),X_m)$ as claimed.
\end{proof}

\section{Weak Solutions: Uniqueness} \label{sec3}

The final step of the proof of Theorem~\ref{ThmP1} is the uniqueness of mass-conserving weak solutions to \eqref{a1} on $[0,\infty)$ having a finite moment of sufficiently high order and the proof given below follows quite closely the approach developed in \cite{EMRR05, LaMi04, Norr99}.

\begin{proof}[Proof of Theorem~\ref{ThmP1}~(c): Uniqueness] 
Let $f^{in}$ be an initial condition in $X_{m_0}\cap X_{2\lambda-\alpha}$ and two mass-conserving weak solutions $f_1$ and $f_2$ to \eqref{a1} on $[0,\infty)$ which satisfy
\begin{equation}
\mathcal{M}_{2\lambda-\alpha}(T) := \sup_{t\in [0,T]}\{M_{2\lambda-\alpha}(f_1(t))\} + \sup_{t\in [0,T]}\{M_{2\lambda-\alpha}(f_2(t))\} < \infty\ , \qquad T>0\ . \label{u1}
\end{equation}
Observe that the existence of at least such a solution is guaranteed by Theorem~\ref{ThmP1}~(a)-(b) as $2\lambda-\alpha> 1+\lambda-\alpha$. According to Definition~\ref{DefS1} and the choice \eqref{aa11} of $m_0$, there also holds
\begin{equation}
\mathcal{M}_{\alpha}(T) := \sup_{t\in [0,T]}\{M_{\alpha}(f_1(t))\} + \sup_{t\in [0,T]}\{M_{\alpha}(f_2(t))\} < \infty\ , \qquad T>0\ . \label{u2}
\end{equation}
We set $W(x) := x^\alpha + x^\lambda$ for $x\ge 0$, $E:= f_1-f_2$, and $\Sigma := \mathrm{sign}(E)$. Let $T>0$. We infer from \eqref{a1} that, for $t\in (0,T)$,
\begin{align}
\frac{\mathrm{d}}{\mathrm{d}t} \int_0^\infty W(x) |E(t,x)|\ \mathrm{d}x & = \frac{1}{2} \int_0^\infty \int_0^\infty K(x,y) \chi_{W \Sigma(t)}(x,y) (f_1+f_2)(t,x) E(t,y)\ \mathrm{d}y\mathrm{d}x \nonumber \\
& \qquad - \int_0^\infty a(y) N_{W\Sigma(t)}(y) E(t,y)\ \mathrm{d}y\ . \label{u3}
\end{align}  
On the one hand, since $\Sigma^2 E = E$, it follows from the convexity of $x\mapsto x^\lambda$ and  the subadditivity of $x\mapsto x^\alpha$ and $x\mapsto x^{\lambda-1}$ that, for $(x,y)\in (0,\infty)^2$, 
\begin{align*}
\chi_{W \Sigma(t)}(x,y) E(t,y) & = \left[ W(x+y) \Sigma(t,x+y) - W(x) \Sigma(t,x) - W(y) \Sigma(t,y) \right] E(t,y) \\
& \le \left[ W(x+y) + W(x) - W(y) \right] |E(t,y)| \\
& \le \left[ x^\alpha + \lambda x (x+y)^{\lambda-1} + x^\alpha + x^\lambda \right] |E(t,y)| \\
& \le \left[ 2 x^\alpha + (\lambda+1) x^\lambda +\lambda xy^{\lambda-1} \right] |E(t,y)| \\
& \le 3 \left[ x^\alpha + x^\lambda + x y^{\lambda-1} \right] |E(t,y)| \ .
\end{align*}
Therefore, by \eqref{aa3}, 
\begin{align*}
& \hspace{-0.5cm} K(x,y) \chi_{W \Sigma(t)}(x,y) E(t,y) \\
& \le 3 K_0 \left( x^\alpha + x^\lambda + xy^{\lambda-1} \right) \left( x^\alpha y^{\lambda-\alpha} + x^{\lambda-\alpha} y^\alpha \right) |E(t,y)| \\
& \le 3 K_0 \left( x^{2\alpha} y^{\lambda-\alpha} + x^{\alpha+\lambda} y^{\lambda-\alpha} + x^{1+\alpha} y^{2\lambda-\alpha-1} \right) |E(t,y)| \\
& \qquad + 3 K_0 \left( x^{\lambda} y^{\alpha} + x^{2\lambda-\alpha} y^{\alpha} + x^{\lambda+1-\alpha} y^{\lambda+\alpha-1} \right) |E(t,y)| \ .
\end{align*}
Observing that all the exponents in the powers of $y$ range in $[\alpha,\lambda]$ and all the exponents in the powers of $x$ range in $[\alpha,2\lambda-\alpha]$ by \eqref{aa1}, it further follows from several applications of Young's inequality that
\begin{align*}
& \hspace{-0.5cm} K(x,y) \chi_{W \Sigma(t)}(x,y) E(t,y) \\
& \le 3 K_0 \left( x^{2\alpha} + x^{\alpha+\lambda} + x^{1+\alpha} +  x^{\lambda} + x^{2\lambda-\alpha} + x^{\lambda+1-\alpha} \right) W(y) |E(t,y)| \\
& \le 18 K_0 \left( x^\alpha + x^{2\lambda-\alpha} \right) W(y) |E(t,y)|\ .
\end{align*}
On the other hand, we infer from \eqref{aa4}, \eqref{aa5}, \eqref{aa6}, and \eqref{aa8} that, for $y>0$, 
\begin{align*}
N_{W \Sigma(t)}(y) E(t,y) & = W(y) |E(t,y)| - \int_0^y W(x) \Sigma(t,x) b(x,y) E(t,y)\ \mathrm{d}x \\
& \ge \left( W(y) - \int_0^y W(x) b(x,y)\ \mathrm{d}x \right) |E(t,y)| \\
& \ge \left( \int_0^1 \left[ y^\alpha z + y^\lambda z - (yz)^\alpha - (yz)^\lambda \right] B(z)\ \mathrm{d}z \right) |E(t,y)| \\
& \ge - \mathfrak{b}_{\alpha,1} y^\alpha |E(t,y)|\ , 
\end{align*}
Inserting the previous estimates in \eqref{u3} and using \eqref{aa2}, \eqref{u1}, \eqref{u2}, and the inequality $y^{\alpha+\lambda-1} \le W(y)$, $y>0$, lead us to 
\begin{align*}
\frac{\mathrm{d}}{\mathrm{d}t} \int_0^\infty W(x) |E(t,x)|\ \mathrm{d}x & \le 9 K_0 \int_0^\infty \int_0^\infty \left( x^\alpha + x^{2\lambda-\alpha} \right) W(y) (f_1+f_2)(t,x) |E(t,y)|\ \mathrm{d}y\mathrm{d}x \\
& \qquad + a_0 \mathfrak{b}_{\alpha,1} \int_0^\infty y^{\alpha+\lambda-1} |E(t,y)|\ \mathrm{d}y \\
& \le \left[ 9 K_0 \left( \mathcal{M}_\alpha(T) + \mathcal{M}_{2\lambda-\alpha}(T) \right) + a_0 \mathfrak{b}_{\alpha,1} \right] \int_0^\infty W(x) |E(t,x)|\ \mathrm{d}x\ .
\end{align*}
Integrating with respect to time gives
\begin{equation*}
\int_0^\infty W(x) |E(t,x)|\ \mathrm{d}x \le e^{\left[ 9 K_0 \left( \mathcal{M}_\alpha(T) + \mathcal{M}_{2\lambda-\alpha}(T) \right) + a_0 \mathfrak{b}_{\alpha,1}  \right]t} \int_0^\infty W(x) |E(0,x)|\ \mathrm{d}x = 0
\end{equation*}
for $t\in [0,T]$. Hence, $f_1(t)=f_2(t)$ for all $t\in [0,T]$, and the claimed uniqueness follows as $T$ is arbitrary in $(0,\infty)$.
\end{proof}


\bibliographystyle{siam}
\bibliography{GlobalCFCritical}

\end{document}